\documentclass[10pt]{elsarticle}

\usepackage{amssymb,amsmath}
\usepackage{enumerate} %enumerate
\usepackage{enumitem} %format enumerate

\def\vertex{\circle*{1.8}}
\def\ball{\circle{8}}

\newtheorem{theorem}{Theorem}[section]
\newtheorem{lemma}[theorem]{Lemma}

\newenvironment{proof}[1][]%
{\noindent {\setcounter{equation}{0}\it Proof.
}{#1}{}}{\hfill$\Box$\vspace{2ex}}

\def\longbox#1{\parbox{0.85\textwidth}{#1}}

\begin{document}

\title{$4$-coloring $(P_6,\mbox{bull})$-free graphs}

\author{Fr\'ed\'eric Maffray\corref{cor2}\fnref{STINT}}
%\ead{frederic.maffray@g-scop.grenoble-inp.fr}

\author{Lucas Pastor\corref{cor1}\fnref{STINT}}
%\ead{lucas.pastor@g-scop.grenoble-inp.fr}

\cortext[cor2]{CNRS.}
\cortext[cor1]{Corresponding author.}

\fntext[STINT]{The authors are partially supported by ANR project
STINT (ANR-13-BS02-0007).}

\address{Laboratoire G-SCOP, Universit\'e Grenoble-Alpes, Grenoble,
France}

\date{\today}

\begin{frontmatter}
\begin{abstract}
We present a polynomial-time algorithm that determines whether a graph
that contains no induced path on six vertices and no bull (the graph
with vertices $a,b,c,d,e$ and edges $ab,bc,cd,be,ce$) is
$4$-colorable.  We also show that for any fixed $k$ the $k$-coloring
problem can be solved in polynomial time in the class of
$(P_6,\mbox{bull},\mbox{gem})$-free graphs.

\end{abstract}

\end{frontmatter}

\section{Introduction}

For any integer $k$, a \emph{$k$-coloring} of a graph $G$ is a mapping
$c:V(G)\rightarrow\{1,\ldots,k\}$ such that any two adjacent vertices
$u,v$ in $G$ satisfy $c(u)\neq c(v)$.  A graph is \emph{$k$-colorable}
if it admits a $k$-coloring.  The \emph{chromatic number} $\chi(G)$ of
a graph $G$ is the smallest integer $k$ such that $G$ is
$k$-colorable.  Determining whether a graph is $k$-colorable is
NP-complete for each fixed $k\ge 3$ \cite{Karp,GJS}.

For any integer $\ell$ we let $P_\ell$ denote the path on $\ell$
vertices and $C_\ell$ denote the cycle on $\ell$ vertices.  Given a
family of graphs ${\cal F}$, a graph $G$ is \emph{${\cal F}$-free} if
no induced subgraph of $G$ is isomorphic to a member of ${\cal F}$;
when ${\cal F}$ has only one element $F$ we say that $G$ is $F$-free.
Recently several authors have considered the following question: {\it
What is the complexity of determining whether a $P_\ell$-free graph is
$k$-colorable?} Here is a summary of what is known so far (for more
details see \cite{BGPS,CMSZ,Huang}).
\begin{itemize}
     \itemsep=0em
\item
For $\ell\le 4$ the problem is solved by the fact that the chromatic
number of any $P_4$-free graph can be computed in polynomial time
\cite{CLS}.
\item
For $\ell=5$, Ho\`ang et al.~\cite{HKLSS} proved that the problem of
$k$-coloring $P_5$-free graphs is polynomially solvable for every
fixed $k$.
\item
For $k=3$, there are polynomial-time algorithms for $3$-coloring
$P_6$-free graphs due to Randerath and Schiermeyer \cite{RS} and later
Broersma~et~al.~\cite{BFGP}, and more recently for $3$-coloring
$P_7$-free graphs due to Chudnovsky et al.~\cite{CMZ1,CMZ2}.
\item
On the other hand, Huang \cite{Huang} showed that the problem is
NP-complete when either $k\ge 4$ and $\ell\ge 7$ or $k\ge 5$ and
$\ell\ge 6$.
\end{itemize}
Hence the cases whose complexity status is still unknown are when
$k=3$ and $\ell\ge 8$ and when $k=4$ and $\ell=6$.  For $k=4$ and
$\ell=6$, Chudnovsky et al.~\cite{CMSZ} gave a polynomial-time
algorithm that decides if a $(P_6,C_5)$-free graph is $4$-colorable;
Huang~\cite{Huang} gave a polynomial-time algorithm that decides if a
$(P_6,\text{banner})$-free graph is $4$-colorable (where the banner is
the graph that consists of a $C_4$ plus a vertex with one neighbor in
the $C_4$); and Brause et al.~\cite{Brause} gave a polynomial-time
algorithm that determines the $4$-colorability of $(P_6,\mbox{bull},
Z_1)$-free graphs and $(P_6,\mbox{bull}, \mbox{kite})$-free graphs,
where the \emph{bull} is the graph with vertices $a,b,c,d,e$ and edges
$ab,bc,cd,be,ce$ (see Figure~\ref{fig:bg}), and $Z_1$ and the kite are
two other graphs on five vertices.

We will generalize the latter to all
$(P_6,\mbox{bull})$-free graphs.  The \emph{gem} is the graph with
vertices $v_1, \ldots, v_5$ and edges $v_1v_2$, $v_2v_3$, $v_3v_4$ and
$v_5v_i$ for all $i\in\{1,2,3,4\}$ (see Figure~\ref{fig:bg}).  Our
main results are the following.

\begin{theorem}\label{thm:main}
There is a polynomial time algorithm that determines whether a
$(P_6,\mbox{bull})$-free graph $G$ is $4$-colorable, and if it is,
produces a $4$-coloring of $G$.
\end{theorem}

\begin{theorem}\label{thm:kp6bgf}
For any fixed $k$, there is a polynomial-time algorithm that
determines if a $(P_6,\mbox{bull},\mbox{gem})$-free graph is
$k$-colorable, and if it is, produces a $k$-coloring of $G$.
\end{theorem}

\begin{figure}[ht]
\unitlength=0.08cm
\thicklines
\begin{center}
\begin{tabular}{cc}

\begin{picture}(24,20) % BULL
       % vertices
\multiput(6,6)(12,0){2}{\vertex}
\multiput(0,18)(24,0){2}{\vertex}
\put(12,15){\vertex}
       % edges
\put(6,6){\line(1,0){12}}
\put(18,6){\line(1,2){6}} \put(6,6){\line(-1,2){6}}
% \put(12,15){\line(4,1){12}} \put(12,15){\line(-4,1){12}}
\put(12,15){\line(2,-3){6}}\put(12,15){\line(-2,-3){6}}
\put(9,-4){Bull}
       % end
\end{picture}

\quad & \quad
\begin{picture}(24,20) % GEM
       % vertices
\multiput(6,6)(12,0){2}{\vertex}
\multiput(0,18)(24,0){2}{\vertex}
\put(12,15){\vertex}
       % edges
\put(6,6){\line(1,0){12}}
\put(18,6){\line(1,2){6}} \put(6,6){\line(-1,2){6}}
\put(12,15){\line(4,1){12}} \put(12,15){\line(-4,1){12}}
\put(12,15){\line(2,-3){6}}\put(12,15){\line(-2,-3){6}}
\put(9,-4){Gem}
       % end
\end{picture}
\\
% \hline
\end{tabular}
\end{center}
\caption{The bull and the gem.}
\label{fig:bg}
\end{figure}
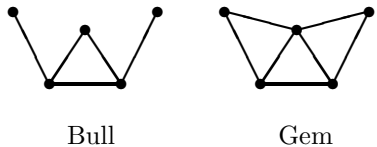

Our paper is organised as follows.  In the rest of this section we
recall some terminology and notation and we show that in order to
prove Theorem~\ref{thm:main} it suffices to prove it for graphs that
satisfy certain restrictions.  In Section~2 we show that the problem
reduces to graphs that do not contain certain special graphs.  In
Section~3 we consider the case when the graph has no gem.  In that
case we show that the graph either is perfect or has bounded
clique-width, from which it follows that the $k$-colorability can be
determined in polynomial time for any fixed $k$, thus proving
Theorem~\ref{thm:kp6bgf}.  Finally in Section~4 we consider the case
when the graph contains a gem and none of the special graphs, and we
deduce a structural description of the graph which we can use to solve
$4$-colorability problem directly, proving Theorem~\ref{thm:main}.

\medskip

Let us recall some definitions and notation.  Let $G$ be a graph.  For
each $v\in V(G)$, we denote by $N_G(v)$ the set of vertices adjacent
to $v$ (the \emph{neighbors} of $v$) in $G$, and when there is no
ambiguity we simply write $N(v)$.  For any subset $S$ of $V(G)$ we
write $N_S(v)$ instead of $N(v)\cap S$; and for a subgraph $H$ we
write $N_H(v)$ instead of $N_{V(H)}(v)$.  We say that a vertex $v$ is
\emph{complete} to $S$ if $v$ is adjacent to every vertex in $S$, and
that $v$ is \emph{anticomplete} to $S$ if $v$ has no neighbor in $S$.
For two sets $S,T\subseteq V(G)$ we say that $S$ is complete to $T$ if
every vertex of $S$ is adjacent to every vertex of $T$, and we say
that $S$ is anticomplete to $T$ if no vertex of $S$ is adjacent to any
vertex of $T$.  For $S\subseteq V(G)$ we denote by $G[S]$ the induced
subgraph of $G$ with vertex-set $S$, and we denote by $N(S)$ the set
$\{v\in V(G)\setminus S\mid$ $v$ has a neighbor in $S\}$.  The
complement of $G$ is denoted by $\overline{G}$.

\medskip

Let $\omega(G)$ denote the maximum size of a clique in
$G$.  A graph $G$ is \emph{perfect} if every induced subgraph $H$ of
$G$ satisfies $\chi(H)=\omega(H)$.   By the Strong Perfect Graph
Theorem \cite{CRST}, a graph is perfect if and only if it contains no
$C_\ell$ and no $\overline{C_\ell}$ for any odd $\ell\ge 5$.

\subsection{Homogeneous sets and modules}\label{sec:mod}

A \emph{homogeneous set} is a set $S\subseteq V(G)$ such that every
vertex in $V(G)\setminus S$ is either complete to $S$ or anticomplete
to $S$.  A homogeneous set is \emph{proper} if it contains at leat two
vertices and is different from $V(G)$.  A graph is \emph{prime} if it
has no proper homogeneous set.  A \emph{module} is a homogeneous set
$S$ such that every homogeneous set $S'$ satisfies either $S'\subseteq
S$ or $S\subseteq S'$ or $S\cap S'=\emptyset$.  In particular $V(G)$
is a module and every singleton $\{v\}$ $(v\in V(G)$) is a module.
The theory of modular decomposition (the study of the modules of a
graph) is a rich one, starting from the seminal work of Gallai
\cite{Gallai}.  We mention here only the results we will use.  We say
that a module $S$ is \emph{maximal} if $S\neq V(G)$ and there is no
module $S'$ such that $S\subset S'\subset V(G)$ (with strict
inclusion).
\begin{itemize}
\item
Any graph $G$ has at most $2|V(G)|$ modules, and they can be produced
by an algorithm of time complexity $O(|V(G)|+|E(G)|)$ \cite{CHPT}.
\item
If both $G$ and $\overline{G}$ are connected, then $G$ has a least
four maximal modules and they form a partition of $V(G)$, and every
homogeneous set of $G$ different from $V(G)$ is included in a maximal
module; moreover, the induced subgraph $G'$ of $G$ obtained by picking
one vertex from each maximal module of $G$ is a prime graph.
\end{itemize}

Here we will say that a graph $G$ is \emph{quasi-prime} if every
proper homogeneous set of $G$ is a clique.  We say that two vertices
$u,v$ are \emph{twins} if $\{u,v\}$ is a homogeneous set.  Hence in a
quasi-prime graph every homogeneous set consists of pairwise adjacent
twins.

\medskip

We let $K_n$ denote the complete graph on $n$ vertices.  The graph
$K_3$ is usually called a \emph{triangle}.  We call \emph{double
wheel} the graph obtained from a $C_5$ by adding two adjacent vertices
$a,b$ with all edges from $a$ and $b$ to the vertices of the $C_5$.
Note that $K_5$ and the double wheel are not $4$-colorable.

\begin{lemma}\label{lem:abc}
In order to prove Theorem~\ref{thm:main} it suffices to prove it for
the $(P_6,\mbox{bull})$-free graphs $G$ that satisfy the following
properties:
\begin{itemize}
\item[(a)]
$G$ is connected and $\overline{G}$ is connected.
\item[(b)]
$G$ is quasi-prime.
\item[(c)]
$G$ is $K_5$-free and double-wheel-free.
\end{itemize}
\end{lemma}
\begin{proof}
Assume that we want to determine whether a $(P_6,\mbox{bull})$-free
graph $G$ is $4$-colorable.

\medskip

(a) If $G$ is not connected we can examine each component of $G$
separately.  Now suppose that $\overline{G}$ is not connected.  So
$V(G)$ can be partitioned into two non-empty sets $V_1$ and $V_2$ that
are complete to each other.  It follows that $\chi(G)=
\chi(G[V_1])+\chi(G[V_2])$.  A necessary condition for $G$ to be
$4$-colorable is that $G[V_i]$ is $3$-colorable for each $i=1,2$.
Using the algorithms from \cite{BFGP} or \cite{RS} we can test whether
$G[V_1]$ and $G[V_2]$ are $3$-colorable.  If any of them is not
$3$-colorable we declare that $G$ is not $4$-colorable and stop.  If
each $G[V_i]$ is $3$-colorable, we can determine the value of
$\chi(G[V_i])$ by further testing whether $G[V_i]$ is either edgeless
or bipartite.  Hence we can determine if $G$ is $4$-colorable (and if
it is, give a $4$-coloring) in polynomial time.

\medskip

(b) Suppose that $G$ is not quasi-prime.  So $G$ has a homogeneous set
$X$ that is not a clique and $X\neq V(G)$.  Since $G$ and
$\overline{G}$ are connected $X$ is included in a maximal module.
Hence let us consider any maximal module $M$ of $G$ that is not a
clique.  We know that $M\neq V(G)$, so the set $N(M)$ is not empty,
and $N(M)$ is complete to $M$.  So a necessary condition for $G$ to be
$4$-colorable is that $G[M]$ is $3$-colorable.  Using the algorithms
from \cite{BFGP} or \cite{RS} we can determine whether $G[M]$ is
$3$-colorable or not.  If it is not we declare that $G$ is not
$4$-colorable and stop.  If $G[M]$ is $3$-colorable, we can determine
the value of $\chi(G[M])$ by further testing whether $G[M]$ is either
edgeless or bipartite.  Then we build a new graph $G'$ from $G$ by
removing $M$ and adding a clique $K_M$ of size $\chi(G[M])$ with edges
from every vertex of $K_M$ to every vertex in $N(M)$ and no other
edge.  Thus $K_M$ is a homogeneous set in $G'$, with the same
neighborhood as $M$ in $G$.  We observe that:
\begin{equation}\label{gppbf}
\mbox{$G'$ is $P_6$-free and bull-free.}
\end{equation}
Proof: If $G'$ has an induced subgraph $H$ that is either a $P_6$ or a
bull, then $H$ must contain a vertex $v$ from $K_M$ (because
$G'\setminus K_M=G\setminus M$), and $H$ does not contain two vertices
from $K_M$ since $H$ has no twins.  Then, replacing $v$ with any
vertex from $M$ yields an induced $P_6$ or bull in $G$, a
contradiction.  So (\ref{gppbf}) holds.

We repeat this operation for every maximal module of $G$ that is not a
clique.  Hence we obtain a graph $G''$ where every such module $M$ has
been replaced with a clique $K_M$, and, by the same argument as in
(\ref{gppbf}), $G''$ is $P_6$-free and bull-free.  For convenience we
set $K_L=L$ whenever $L$ is a maximal module of $G$ that is a clique.
We observe that:
\begin{equation}\label{gdpqp}
\mbox{$G''$ is quasi-prime.}
\end{equation}
Proof: Suppose that $G''$ has a homogeneous set $Y''$ that is not a
clique, and $Y''\neq V(G'')$.  Let $A''=N_{G''}(Y'')$ and
$B''=V(G'')\setminus (Y''\cup A'')$.  For each vertex $x\in V(G'')$
let $M_x$ be the maximal module of $G$ such that $x\in K_{M_x}$.  Let
$Y=\bigcup_{x\in Y''} M_x$, $A=\bigcup_{x\in A''} M_x$ and
$B=\bigcup_{x\in B''} M_x$.  In $G$ the set $Y$ is complete to $A$ and
anticomplete to $B$, and $V(G)=Y\cup A\cup B$.  So $Y$ is a
homogeneous set of $G$, and $Y\neq V(G)$, so there is a maximal module
$L$ of $G$ such that $Y\subseteq L$.  But this implies $Y''\subseteq
K_L$, a contradiction.  So (\ref{gdpqp}) holds.

\begin{equation}\label{ggp}
\mbox{$G$ is $4$-colorable if and only if $G''$ is $4$-colorable.}
\end{equation}
Proof: Let $c$ be a $4$-coloring of $G$.  For each maximal module $M$
of $G$ we have $|c(M)|\ge \chi(G[M])$.  So we can assign to the
vertices of $K_M$ distinct colors from the set $c(M)$.  Doing this for
every $M$ yields a $4$-coloring of $G''$.  Conversely, let $c''$ be a
$4$-coloring of $G''$.  For every maximal module $M$ of $G$, consider
a $\chi(G[M])$-coloring of $M$ and assign to each class of this
coloring one color from the set $c''(K_M)$ (a different color for each
class).  Doing this for every $M$ yields a $4$-coloring of $G$.  So
(\ref{ggp}) holds.

\medskip

The operations performed to construct $G''$ can be done in polynomial
time using modular decomposition \cite{CHPT} and the algorithms from
\cite{RS,BFGP}.  Since the maximal modules of $G$ form a partition of
$V(G)$ their number is $O(|V(G)|)$.  So we can ensure that property
(b) holds through a polynomial time reduction.

\medskip

(c) One can decide in polynomial time whether $G$ contains $K_5$ or
the double wheel, and if it does we stop since these two graphs are
not $4$-colorable.
\end{proof}

The complexity of testing if a $P_6$-free graph on $n$ vertices is
$3$-colorable is $O(n^{\alpha+2})$ in \cite{RS} (where $\alpha$ is the
exponent given by the fast matrix multiplication, $\alpha<2.36$) and
seems to be $O(n^3)$ in \cite{BFGP} using the special dominating set
argument from \cite{vHP}.  Hence the total complexity of the reduction
steps described in the preceding lemma is $O(n^5)$.

%%%%%%
\section{Brooms and magnets}\label{sec:bm}

In this section we prove that if a quasi-prime ($P_6$, bull)-free
graph $G$ contains certain special graphs (called ``magnets''), then
the $4$-colorability of $G$ can be solved in polynomial time using a
reduction to the $2$-list coloring problem.

We first show that if a ($P_6$, bull)-free graph $G$ contains a
certain graph which we call a broom, then either $G$ is not
quasi-prime, or the broom can be extended to subgraphs that will be
convenient to us.

\subsection{Brooms}\label{sec:broom}

A \emph{broom} is a graph with six vertices $v_1, \ldots, v_6$ and
edges $v_1v_2$, $v_2v_3$, $v_3v_4$ and $v_5v_i$ for each
$i\in\{1,2,3,4,6\}$.  See Figure~\ref{fig:BF7}.

Let $F_0$ be the graph with seven vertices $v_1,\ldots,v_7$ and edges
$v_1v_2$, $v_2v_3$, $v_3v_4$, $v_4v_5$, $v_5v_1$, $v_6v_i$ for all
$i\in\{1,\ldots,5\}$, and $v_7v_1$, $v_7v_2$, $v_7v_3$, $v_7v_4$.

\begin{figure}[ht]
\unitlength=0.08cm
\thicklines
\begin{center}
% \begin{tabular}{|c|c|c|c|c|c|}
\begin{tabular}{cc}
\begin{picture}(18,33) % BROOM
       % vertices
\multiput(0,6)(6,0){4}{\vertex}
\multiput(9,15)(0,12){2}{\vertex}
       % edges
\put(0,6){\line(1,0){18}}
\put(9,15){\line(0,1){12}}
\put(0,6){\line(1,1){9}}\put(18,6){\line(-1,1){9}}
\put(6,6){\line(1,3){3}}\put(12,6){\line(-1,3){3}}
\put(1,-4){Broom}
\end{picture}

\quad & \quad

\begin{picture}(24,33) % GRAPH F7
       % vertices
\put(15,6){\vertex}
\multiput(12,12)(6,0){2}{\vertex}
\multiput(6,15)(18,0){2}{\vertex}
\multiput(15,18)(0,6){2}{\vertex}
       % edges
\put(15,6){\line(1,1){9}}\put(15,6){\line(-1,1){9}}
\put(15,6){\line(1,2){3}}\put(15,6){\line(-1,2){3}}
\put(12,12){\line(1,0){6}}
\put(12,12){\line(-2,1){6}}\put(18,12){\line(2,1){6}}
\put(15,18){\line(-3,-1){9}}\put(15,18){\line(3,-1){9}}
\put(15,18){\line(-1,-2){3}}\put(15,18){\line(1,-2){3}}
\put(15,18){\line(0,1){6}}
\put(15,24){\line(-1,-1){9}}\put(15,24){\line(1,-1){9}}
\put(13,-4){$F_0$}
       % end
\end{picture}
\end{tabular}
\end{center}
\caption{The broom and the graph $F_0$}\label{fig:BF7}
\end{figure}
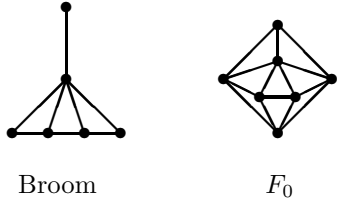

\medskip

The following lemma is an extension of Lemma~2 from \cite{FMP}.

\begin{lemma}\label{lem:broom}
In a bull-free graph $G$, let $\{v_1, \ldots, v_6\}$ be a $6$-tuple
that induces a broom, with edges $v_1v_2$, $v_2v_3$, $v_3v_4$ and
$v_5v_i$ for each $i\in\{1,2,3,4,6\}$.  Then one of the following
holds:
\begin{itemize}
\item
$G$ has a proper homogeneous set that contains $\{v_1, v_2, v_3,
v_4\}$.
\item
There is a vertex $z$ in $V(G)\setminus\{v_1, \ldots, v_6\}$ that is
complete to $\{v_1, v_4, v_6\}$ and anticomplete to $\{v_2, v_3,
v_5\}$.
\item
There are two non-adjacent vertices $z,t$ in $V(G)\setminus\{v_1,
\ldots, v_6\}$ such that $z$ is complete to $\{v_1, v_4, v_5, v_6\}$
and anticomplete to $\{v_2, v_3\}$ and $t$ is complete to $\{v_1, v_2,
v_3, v_4\}$ and anticomplete to $\{v_5,v_6\}$ (and so $\{v_1, \ldots,
v_5,$ $z, t\}$ induces an $F_0$).
\end{itemize}
\end{lemma}
\begin{proof}
Let us assume that the second and third outcome of the lemma do not
occur.  Let $P=\{v_1, v_2, v_3, v_4\}$ and $R=V(G)\setminus P$.  We
classify the vertices of $R$ as follows; let:
\begin{itemize}
\item
$A=\{x\in R\mid$ $x$ is complete to $P\cup\{v_6\}\}$.
\item
$B=\{x\in R\mid$ $x$ is complete to $P$ and not adjacent to $v_6\}$.
\item
$F=\{x\in R\mid$ $x$ is anticomplete to $P\}$.
\item
$X=\{x\in R\setminus F\mid N(x)\cap P$ is included in either $\{v_1,
v_3\}$ or $\{v_2, v_4\}\}$.
\item
$Y=\{x\in R\setminus (A\cup B)\mid$ $x$ is complete to $\{v_1, v_2\}$
or to $\{v_3, v_4\}\}$.
\item
$Z=\{x\in R \mid$ $N(x)\cap P=\{v_1,v_4\}\}$.
\end{itemize}
Note that $v_5\in A$ and $v_6\in F$.  We observe that:
\begin{equation}\label{part}
\mbox{The sets $A, B, F, X, Y, Z$ form a partition of $R$.}
\end{equation}
Proof: Clearly these sets are pairwise disjoint.  Suppose that there
is a vertex $z$ in $R\setminus (A\cup B\cup F\cup X\cup Y\cup Z)$.
Since $z$ is not in $F$, it has a neighbor in $P$, and up to symmetry
we may assume that $z$ has a neighbor in $\{v_1,v_2\}$, and since $z$
is not in $Y$ it has exactly one neighbor in $\{v_1,v_2\}$.  Now since
$z$ is not in $X$, it must also have a neighbor in $\{v_3,v_4\}$, and
similarly it has exactly one neighbor in $\{v_3,v_4\}$.  Since $z$ is
not in $X\cup Z$, it must be that $N(z)\cap P=\{v_2,v_3\}$; but then
$P\cup\{z\}$ induces a bull.  So (\ref{part}) holds.

\begin{equation}\label{FY}
\mbox{$F$ is anticomplete to $Y$.}
\end{equation}
Proof: Suppose that there are adjacent vertices $f\in F$ and $y\in Y$.
Up to symmetry $y$ is complete to $\{v_1,v_2\}$.  Then $y$ must be
adjacent to $v_3$, for otherwise $\{f,y,v_1,v_2,v_3\}$ induces a bull,
and then to $v_4$, for otherwise $\{f,y,v_2,v_3,v_4\}$ induces a bull.
But then $y$ should be in $A\cup B$, not in $Y$.  So (\ref{FY}) holds.

\begin{equation} \label{ABX}
\mbox{$A\cup B$ is complete to $X$.}
\end{equation}
Proof: Suppose that there are non-adjacent vertices $a\in A\cup B$ and
$x\in X$.  Up to symmetry $x$ has exactly one neighbor in
$\{v_1,v_2\}$.  Then $x$ must be adjacent to $v_4$, for otherwise
$\{x, v_1, v_2, a, v_4\}$ induces a bull.  So $x$ is not adjacent to
$v_3$ and, by a symmetric argument, $x$ must be adjacent to $v_1$.
But then $x$ should be in $Z$, not in $X$.  So (\ref{ABX}) holds.

\begin{equation}\label{AYZ}
\mbox{$A$ is complete to $Y\cup Z$.}
\end{equation}
Proof: Suppose that there are non-adjacent vertices $a$ and $y\in
Y\cup Z$.  Suppose that $y\in Y$, say $y$ is complete to
$\{v_1,v_2\}$.  By (\ref{FY}), $y$ is not adjacent to $v_6$.  Then $y$
must be adjacent to $v_3$, for otherwise $\{y,v_2,v_3,a,v_6\}$ induces
a bull, and then to $v_4$, for otherwise $\{y,v_3,v_4,a,v_6\}$ induces
a bull.  But then $y$ should be in $A\cup B$, not in $Y$.  Now suppose
that $y\in Z$.  Then $y$ is adjacent to $v_6$, for otherwise
$\{y,v_1,v_2, a,v_6\}$ induces a bull.  But then we obtain the second
outcome of the lemma, a contradiction.  So (\ref{AYZ}) holds.

\medskip

Let $B'$ be the set of vertices $b$ in $B$ for which there exists in
$\overline{G}$ a chordless path $b_0$-$b_1$-$\cdots$-$b_k$ ($k\ge 1$)
such that $b_0\in Y\cup Z$, $b_1,\ldots, b_k\in B$ and $b_k=b$.  Such
a path will be called a $B'$-path for $b$.
\begin{equation}\label{BmB}
\mbox{$B\setminus B'$ is complete to $Y\cup Z\cup B'$.}
\end{equation}
This follows directly from the definition of $B'$.

\begin{equation}\label{ABp}
\mbox{$A$ is complete to $B'$.}
\end{equation}
Proof: Consider any $a\in A$ and $b\in B'$.  Let
$b_0$-$b_1$-$\cdots$-$b_k$ be a $B'$-path for $b$, as above.  By
(\ref{AYZ}), $a$ is adjacent to $b_0$.  Pick any $v_h$ in $P\cap
N(b_0)$.  First suppose that $b_0$ is not adjacent to $v_6$.  Then for
each $i\ge 1$ and by induction, $a$ is adjacent to $b_i$, for
otherwise $\{b_i, v_h, b_{i-1}, a, v_6\}$ induces a bull.  Hence $a$
is adjacent to $b$.  Now suppose that $b_0$ is adjacent to $v_6$; by
(\ref{FY}), this means that $b_0\in Z$.  Then $a$ must be adjacent to
$b_1$, for otherwise we obtain the third outcome of the lemma (where
$b_0,b_1$ play the role of $z,t$).  Then for each $i\ge 2$ and by
induction, $a$ is adjacent to $b_i$, for otherwise $\{b_i, b_{i-2},
v_6, a, b_{i-1}\}$ induces a bull.  Hence $a$ is adjacent to $b$.  So
(\ref{ABp}) holds.

\medskip

Let $F'$ be the set of vertices in the components of $F$ that have
a neighbor in $X\cup Z$.
\begin{equation}\label{ABmFXZ}
\mbox{$A\cup (B\setminus B')$ is complete to $F'$.}
\end{equation}
Proof: Consider any $a\in A\cup (B\setminus B')$ and $f\in F'$.  By
the definition of $F'$ there is a chordless path $f_0$-$\cdots$-$f_k$
with $f_0\in X\cup Z$, $f_1, \ldots, f_k\in F'$ and $f_k=f$.  By
(\ref{ABX}), (\ref{AYZ}) and (\ref{BmB}), $a$ is adjacent to $f_0$.
Since $f_0\in X\cup Z$, there are non-adjacent vertices $v,v'\in P$
such that $f_0$ is adjacent to $v$ and not to $v'$.  Then $a$ is
adjacent to $f_1$, for otherwise $\{f_1,f_0,v,a,v'\}$ induces a bull.
Then for each $i\ge 2$ and by induction, $a$ is adjacent to $f_i$, for
otherwise $\{f_i, f_{i-1}, f_{i-2}, a, v'\}$ induces a bull.  Hence
$a$ is adjacent to $f$.  So (\ref{ABmFXZ}) holds.

\begin{equation}\label{BpF}
\mbox{$B'$ is anticomplete to $F\setminus F'$.}
\end{equation}
Proof: Consider any $b\in B'$ and $f\in F\setminus F'$, and take a
$B'$-path $b_0$-$\cdots$-$b_k$ for $b$ as above.  Vertex $b_0$ is not
adjacent to $f$ by (\ref{FY}) (if $b_0\in Y$) or by the definition of
$F'$ (if $b_0\in Z$).  There exist two adjacent vertices $v_j,v_{j+1}$
of $P$ such that $b_0$ is adjacent to exactly one of them.  Vertex $f$
is not adjacent to $b_1$, for otherwise $\{b_0, v_j, v_{j+1}, b_1,
f\}$ induces a bull.  Then for each $i\ge 2$ and by induction, $f$ is
not adjacent to $b_i$, for otherwise $\{f, b_i, b_{i-2}, v, b_{i-1}\}$
induces a bull, where $v$ is any vertex in $P\cap N(b_0)$.  Hence $f$
is not adjacent to $b$.  So (\ref{BpF}) is proved.

\medskip

Now let $H=P\cup X\cup Y\cup Z\cup F'\cup B'$.  By (\ref{part}),
$V(G)$ is partitioned into the three sets $H$, $A\cup(B\setminus B')$
and $F\setminus F'$.  It follows from the definition of the sets and
Claims (\ref{ABX})--(\ref{BpF}) that $H$ is complete to $A\cup
(B\setminus B')$ and anticomplete to $F\setminus F'$, and we know that
$A\cup(B\setminus B')\neq\emptyset$ since $v_5\in A$.  So $H$ is a
homogeneous set that contains $\{v_1,v_2,v_3,v_4\}$, and it is proper
since it does not contain $v_5$.
\end{proof}

%%%
\subsection{Magnets}

We recall the variant of the coloring problem known as \emph{list
coloring}, which is defined as follows.  Every vertex $v$ of a graph
$G$ has a list $L(v)$ of allowed colors; then we want to know whether
the graph admits a coloring $c$ such that $c(v)\in L(v)$ for all $v$.
When all lists have size at most $2$ we call it a $2$-list coloring
problem; it is known that such a problem can be solved in linear time
in the size of the input (the number of lists), as it is reducible to
the $2$-satisfiability of Boolean formulas, see \cite{APT}.

Let us say that a subgraph $F$ of $G$ is a \emph{magnet} if every
vertex $x$ in $G\setminus F$ has two neighbors $u,v\in V(F)$ such that
$uv\in E(F)$.

\begin{lemma}\label{lem:magnet}
If a graph $G$ contains a magnet of bounded size, the $4$-coloring
problem can be solved on $G$ in linear time.
\end{lemma}
\begin{proof}
Let $F$ be a magnet in $G$.  We try every $4$-coloring of $F$.  Since
$F$ has bounded size there is a bounded number of possibilities.  We
try to extend the coloring to the rest of the graph as a list coloring
problem.  Every vertex $v$ in $G\setminus F$ has a list $L(v)$ of
available colors, namely the set $\{1,2,3,4\}$ minus the colors
assigned to the neighbors of $v$ in $F$.  Since $F$ is a magnet every
list has size at most $2$.  So coloring $G\setminus F$ is a $2$-list
coloring problem, which can be solved in linear time.
\end{proof}

In a graph $G$, let $\sim$ be the relation defined on the set $E(G)$
by putting $e\sim f$ if and only if $e$ and $f$ have a common vertex
and $e\cup f$ induces a $P_3$ in $G$.  We say that $G$ is
\emph{$P_3$-connected} if it is connected and for any two edges
$e,f\in E(G)$ there is a sequence $e_0$, $e_1$, \ldots, $e_k$ of edges
of $G$ such that $e_0=e$, $e_k=f$, and for all $i\in\{0,\ldots,k-1\}$
$e_i\sim e_{i+1}$.  (In other words, $G$ is $P_3$-connected if it is
connected and $E(G)$ is the unique class of the equivalence closure of
$\sim$.)

\begin{lemma}\label{lem:P3conn}
Let $G$ be a bull-free graph and let $F$ be a $P_3$-connected induced
subgraph of $G$.  Suppose that there are adjacent vertices $x,y$ in
$G\setminus F$ such that $x$ is anticomplete to $F$, and $y$ has two
adjacent neighbors in $F$.  Then $y$ is complete to $F$.
\end{lemma}
\begin{proof}
Let $a,b$ be two adjacent neighbors of $y$ in $F$.  Suppose that $y$
has a non-neighbor $c$ in $F$.  Since $F$ is $P_3$-connected, there is
a sequence $e_0$, $e_1$, \ldots, $e_k$ of edges of $F$ such that
$e_0=\{a,b\}$, $e_k$ contains $c$, and for all $i\in\{0,\ldots,k-1\}$
the edges $e_i$ and $e_{i+1}$ have a common vertex and $e_i\cup
e_{i+1}$ induces a $P_3$.  Then there is an integer $i$ such that $y$
is complete to the two ends of $e_i$ and not complete to the two ends
of $e_{i+1}$, say $e_i=uv$ and $e_{i+1}=vw$; but then $\{x,y,u, v,
w\}$ induces a bull, a contradiction.
\end{proof}

We define six more graphs as follows (see Figure~\ref{fig:F1F6}):
\begin{itemize}
\itemsep=0em
\item
Let $F_1$ be the graph with vertices $v_1,\ldots,v_6$ and edges
$v_1v_2$, $v_2v_3$, $v_3v_4$, $v_4v_5$ and $v_6v_i$ for all
$i\in\{1,\ldots,5\}$.
\item
Let $F_2$ be the graph obtained from $F_1$ by adding the edge
$v_1v_5$.
\item
Let $F_3$ be the graph with vertices $v_1,\ldots,v_6$ and edges
$v_1v_2$, $v_1v_3$, $v_2v_3$, $v_2v_4$, $v_3v_4$, $v_3v_5$, $v_4v_5$,
$v_4v_6$ and $v_5v_6$.
\item
Let $F_4$ be the graph obtained from $F_3$ by adding the edge
$v_1v_6$.
\item
Let $F_5=\overline{C_6}$.
\item
Let $F_6$ be the graph with vertices $v_1,\ldots,v_7$ and edges
$v_1v_2$, $v_2v_3$, $v_3v_4$, $v_4v_5$, $v_5v_1$, $v_6v_1$, $v_6v_2$,
$v_6v_3$, $v_6v_5$, $v_7v_2$, $v_7v_3$, $v_7v_4$, $v_7v_5$ and
$v_7v_6$.
\end{itemize}

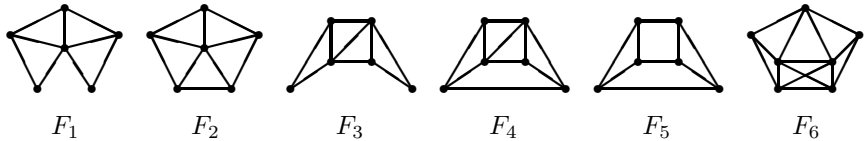
\begin{figure}[ht]
\unitlength=0.06cm
\thicklines
\begin{center}
% \begin{tabular}{|c|c|c|c|c|c|}
\begin{tabular}{cccccc}

%        \hline

\begin{picture}(24,30) % GRAPH F1
       % vertices
\multiput(6,6)(12,0){2}{\vertex}
\multiput(0,18)(24,0){2}{\vertex}
\multiput(12,15)(0,9){2}{\vertex}
       % edges
% \put(6,12){\line(1,0){12}}
\put(12,15){\line(0,1){9}}
\put(18,6){\line(1,2){6}} \put(6,6){\line(-1,2){6}}
\put(12,15){\line(4,1){12}} \put(12,15){\line(-4,1){12}}
\put(12,24){\line(2,-1){12}} \put(12,24){\line(-2,-1){12}}
\put(12,15){\line(2,-3){6}}\put(12,15){\line(-2,-3){6}}
\put(9,-4){$F_1$}
       % end
\end{picture}

&
\begin{picture}(24,30) % GRAPH F2
       % vertices
\multiput(6,6)(12,0){2}{\vertex}
\multiput(0,18)(24,0){2}{\vertex}
\multiput(12,15)(0,9){2}{\vertex}
       % edges
\put(6,6){\line(1,0){12}}
\put(12,15){\line(0,1){9}}
\put(18,6){\line(1,2){6}} \put(6,6){\line(-1,2){6}}
\put(12,15){\line(4,1){12}} \put(12,15){\line(-4,1){12}}
\put(12,24){\line(2,-1){12}} \put(12,24){\line(-2,-1){12}}
\put(12,15){\line(2,-3){6}}\put(12,15){\line(-2,-3){6}}
\put(9,-4){$F_2$}
       % end
\end{picture}

&
\begin{picture}(27,30) % GRAPH F3
       % vertices
\multiput(0,6)(27,0){2}{\vertex}
\multiput(9,12)(9,0){2}{\vertex}
\multiput(9,21)(9,0){2}{\vertex}
       % edges
% \put(0,12){\line(1,0){27}}
\multiput(9,12)(0,9){2}{\line(1,0){9}}
\multiput(9,12)(9,0){2}{\line(0,1){9}}
\put(0,6){\line(3,2){9}} \put(27,6){\line(-3,2){9}}
\put(0,6){\line(3,5){9}} \put(27,6){\line(-3,5){9}}
\put(9,12){\line(1,1){9}}
\put(10,-4){$F_3$}
       % end
\end{picture}

&
\begin{picture}(27,30) % GRAPH F4
       % vertices
\multiput(0,6)(27,0){2}{\vertex}
\multiput(9,12)(9,0){2}{\vertex}
\multiput(9,21)(9,0){2}{\vertex}
       % edges
\put(0,6){\line(1,0){27}}
\multiput(9,12)(0,9){2}{\line(1,0){9}}
\multiput(9,12)(9,0){2}{\line(0,1){9}}
\put(0,6){\line(3,2){9}} \put(27,6){\line(-3,2){9}}
\put(0,6){\line(3,5){9}} \put(27,6){\line(-3,5){9}}
\put(9,12){\line(1,1){9}}
\put(10,-4){$F_4$}
       % end
\end{picture}

&
\begin{picture}(27,30) % GRAPH F5
       % vertices
\multiput(0,6)(27,0){2}{\vertex}
\multiput(9,12)(9,0){2}{\vertex}
\multiput(9,21)(9,0){2}{\vertex}
       % edges
\put(0,6){\line(1,0){27}}
\multiput(9,12)(0,9){2}{\line(1,0){9}}
\multiput(9,12)(9,0){2}{\line(0,1){9}}
\put(0,6){\line(3,2){9}} \put(27,6){\line(-3,2){9}}
\put(0,6){\line(3,5){9}} \put(27,6){\line(-3,5){9}}
\put(10,-4){$F_5$}
       % end
\end{picture}

&
\begin{picture}(24,30) % GRAPH F6
       % vertices
\multiput(6,6)(12,0){2}{\vertex}
\multiput(6,12)(12,0){2}{\vertex}
\multiput(0,18)(24,0){2}{\vertex}
\put(12,24){\vertex}
       % edges
\multiput(6,6)(0,6){2}{\line(1,0){12}}
\put(18,6){\line(1,2){6}} \put(6,6){\line(-1,2){6}}
\multiput(6,12)(6,12){2}{\line(2,-1){12}}
\multiput(6,6)(-6,12){2}{\line(2,1){12}}
\multiput(6,6)(12,0){2}{\line(0,2){6}}
\put(6,12){\line(-1,1){6}} \put(18,12){\line(1,1){6}}
\put(18,12){\line(-1,2){6}}\put(6,12){\line(1,2){6}}

\put(9,-4){$F_6$}
       % end
\end{picture}
\\
% \hline
\end{tabular}
\end{center}
\caption{Graphs $F_1, \ldots, F_6$}
\label{fig:F1F6}
\end{figure}

Recall the graph $F_0$ defined at the beginning of
Section~\ref{sec:broom}.  It is easy to check that each of $F_3$,
$F_4$, $F_5$, $F_6$ and $F_0$ is $P_3$-connected.  (Actually it
follows from \cite{Gallai} that every prime graph is $P_3$-connected.)

\begin{lemma}\label{lem:Fi}
Let $G$ be a quasi-prime bull-free graph that contains no $K_5$ and no
double wheel.  Let $F$ be an induced subgraph of $G$.  Then:
\begin{itemize}
\item
If $F$ is (isomorphic to) $F_0$, then $F$ is a magnet in $G$.
\item
If $G$ is $F_0$-free, and $F$ induces a gem, with vertices $v_1,
\ldots, v_5$ and edges $v_1v_2$, $v_2v_3$, $v_3v_4$ and $v_5v_i$ for
each $i\in\{1,\ldots,4\}$, then either $F$ is a magnet or some vertex
in $G\setminus F$ is complete to $\{v_1,v_4\}$ and anticomplete to
$\{v_2,v_3,v_5\}$.
\item
If $G$ is $F_0$-free, and $F$ is (isomorphic to) any of $F_1$, \ldots,
$F_6$, then $F$ is a magnet in $G$.
\end{itemize}
\end{lemma}
\begin{proof}
We use the same notation as in the definition of $F_0$, $F_1$, \ldots,
$F_6$.

First suppose that $F$ is isomorphic to $F_0$.  Suppose that $F$ is
not a magnet, so there is a vertex $z$ in $G\setminus F$ such that
$N_F(z)$ is a stable set.  We claim that every such vertex satisfies
$N_F(z)=\emptyset$.  For suppose not.  If $z$ is adjacent to $v_1$,
then it is also adjacent to $v_3$, for otherwise $\{z,v_1, v_5,v_6,
v_3\}$ induces a bull, and to $v_4$, for otherwise $\{z,v_1,v_2, v_6,
v_4\}$ induces a bull; but then $N_F(z)$ is not a stable set.  So $z$
is not adjacent to $v_1$, and, by a similar argument (not using
$v_7$), $z$ is not adjacent to any of $v_2$, $v_3$, $v_4$ or $v_5$.
Then $z$ is not adjacent to $v_7$, for otherwise $\{z,v_7,v_3, v_4,
v_5\}$ induces a bull, and also not adjacent to $v_6$, for otherwise
$\{z,v_6,v_5,v_4,v_7\}$ induces a bull.  So the claim holds.  Since
$G$ is connected, there are adjacent vertices $x,y$ in $G\setminus F$
such that $N_F(x)=\emptyset$ and $N_F(y)\neq\emptyset$.  By the same
proof as for the claim, $N_F(y)$ is not a stable set.  Since $F_0$ is
$P_3$-connected, Lemma~\ref{lem:P3conn} implies that $y$ is complete
to $V(F)$.  But then $(V(F)\setminus \{v_7\})\cup\{y\}$ induces a
double wheel, a contradiction.  This proves the first item of the
lemma.

\medskip

Now we prove the second item of the lemma.  Let $F$ have vertices
$v_1, \ldots, v_5$ and edges $v_1v_2$, $v_2v_3$, $v_3v_4$ and $v_5v_i$
for each $i\in\{1,\ldots,4\}$.  Suppose that $F$ is not a magnet; so
there is a vertex $y$ such that $N_F(y)$ is a stable set.  First
suppose that $N_F(y)\neq\emptyset$.  If $y$ is adjacent to $v_5$, then
$F\cup \{y\}$ induces broom.  By Lemma~\ref{lem:broom} and since $G$
is quasi-prime (so $G$ cannot have a homogeneous set that contains the
four vertices of a $P_4$) and $G$ contains no $F_0$, there is a vertex
$z$ complete to $\{v_1,v_4\}$ and anticomplete to $\{v_2,v_3,v_5\}$,
and so the desired result holds.  Now suppose that $y$ is not adjacent
to $v_5$; so, up to symmetry, $y$ has exactly one neighbor in
$\{v_1,v_2\}$.  Then $y$ is adjacent to $v_4$, for otherwise
$\{y,v_1,v_2,v_5,v_4\}$ induces a bull, so $y$ has exactly one
neighbor in $\{v_3,v_4\}$, and by symmetry $y$ is adjacent to $v_1$.
So the desired result holds.  Now suppose that $N_F(y)=\emptyset$.
Since $G$ is connected there is an edge $uv$ such that
$N_F(u)=\emptyset$ and $N_F(v)\neq\emptyset$.  By the preceding
argument we may assume that $N_F(v)$ is not a stable set.  Suppose
that $v$ is adjacent to $v_5$.  Up to symmetry, $v$ is also adjacent
to a vertex $w\in\{v_1,v_2\}$.  Then $v$ is adjacent to $v_4$, for
otherwise $\{u,v,w,v_5,v_4\}$ induces a bull, and, by symmetry, to
$v_1$, and also to $v_2$, for otherwise $\{u,v,v_4,v_5,v_2\}$ induces
a bull, and, by symmetry, to $v_3$.  Hence $\{u,v,v_1,v_2, v_3,v_4\}$
induces a broom, so by Lemma~\ref{lem:broom} there is a vertex $y$
complete to $\{v_1,v_4\}$ and anticomplete to $\{v_2,v_3,v\}$, and so
the desired result holds.  Now suppose that $v$ is not adjacent to
$v_5$.  Then $v$ is adjacent to two adjacent vertices in $\{v_1,v_2,
v_3,v_4\}$, and since $G[\{v_1,v_2, v_3,v_4\}]$ is $P_3$-connected
Lemma~\ref{lem:P3conn} implies that $v$ is complete to $\{v_1,v_2,
v_3,v_4\}$, so $\{u,v,v_1,v_2, v_3,v_4\}$ induces a broom again and we
can conclude as above.

\medskip

Now we prove the third item of the lemma.  First let $F=F_1$, with the
same notation as in the definition.  Suppose that $F$ is not a magnet.
In particular the gem induced by $F\setminus\{v_5\}$ is not a magnet,
so, by the second item of this lemma, there is a vertex $z$ complete
to $\{v_1, v_4\}$ and anticomplete to $\{v_2,v_3,v_6\}$.  If $z$ is
not adjacent to $v_5$, then $\{z, v_4, v_5, v_6, v_2\}$ induces a
bull.  If $z$ is adjacent to $v_5$, then $\{v_1,z,v_5,v_4,v_3\}$
induces a bull, a contradiction.

\medskip

Now let $F=F_2$, with the same notation as in the definition.  Suppose
that $F$ is not a magnet.  For each $i\in\{1,\ldots,5\}$ the gem
induced by $F\setminus\{v_i\}$ is not a magnet, so, by the second item
of this lemma, there is a vertex $z_i$ complete to $\{v_{i-1},
v_{i+1}\}$ and anticomplete to $\{v_{i-2},v_{i+2},v_6\}$.  Then $z_i$
is adjacent to $v_i$, for otherwise $\{z_i, v_{i-1}, v_i, v_6,
v_{i+2}\}$ induces a bull.  The vertices $z_1, \ldots, z_5$ are
pairwise distinct because the sets $N_F(z_i)$ are pairwise different.
Since $G$ contains no $K_5$, the set $\{z_1, \ldots, z_5\}$ is not a
clique, so, up to symmetry, $z_1$ is non-adjacent to either $z_2$ or
$z_3$.  If $z_1$ is not adjacent to $z_2$, then $\{z_1, v_2, z_2, v_3,
v_4\}$ induces a bull.  If $z_1$ is not adjacent to $z_3$, then
$\{z_1, v_5, v_6, v_4, z_3\}$ induces a bull, a contradiction.

\medskip

Now let $F=F_3$.  (When $F$ is $F_4$, $F_5$ or $F_6$ the proof is
similar and we omit the details.)  Suppose that there is a vertex $z$
in $G\setminus F$ such that $N_F(z)$ is a stable set.  We claim that
every such vertex satisfies $N_F(z)=\emptyset$.  For suppose not.  If
$z$ is adjacent to $v_1$, then $z$ is adjacent to $v_5$, for otherwise
$\{z,v_1,v_2,v_3,v_5\}$ induces a bull; but then $\{z,v_5,
v_6,v_4,v_2\}$ induces a bull.  So, and by symmetry, $z$ has no
neighbor in $\{v_1,v_6\}$.  If $z$ has a neighbor in $\{v_2,v_3\}$,
then $\{z,v_2,v_3, v_4,v_6\}$ induces a bull.  So, and by symmetry,
$z$ has no neighbor in $\{v_2,v_3,v_4,v_5\}$.  Thus the claim holds.
(The same claim holds when $F$ is $F_4$, $F_5$ or $F_6$ and we omit
the details.)  Since $G$ is connected, there are adjacent vertices
$x,y$ in $G\setminus F$ such that $N_F(x)=\emptyset$ and
$N_F(y)\neq\emptyset$.  By the same argument as for the claim,
$N_F(y)$ is not a stable set.  Since $F$ is $P_3$-connected,
Lemma~\ref{lem:P3conn} implies that $y$ is complete to $F$.  Note that
$v_1$-$v_3$-$v_4$-$v_6$ is an induced $P_4$ in $F$.  (When $F=F_4$,
use the $P_4$ $v_2$-$v_1$-$v_6$-$v_5$; when $F=F_5$ use any $P_4$ of
$F$; when $F=F_6$ use the $P_4$ $v_1$-$v_2$-$v_3$-$v_4$.)  Then
$\{v_1,v_3,v_4,v_6,y,x\}$ induces a broom, so, since $G$ is
quasi-prime and contains no $F_0$, Lemma~\ref{lem:broom} implies the
existence of a vertex $z$ that is complete to $\{v_1,v_6,x\}$ and
anticomplete to $\{v_3,v_4,y\}$.  Clearly, $z\notin V(F)\cup\{x,y\}$.
Then $z$ is not adjacent to $v_2$, for otherwise $\{x,z,v_1,v_2,v_4\}$
induces a bull; and similarly $z$ is not adjacent to $v_5$; but then
$\{z,v_1,v_2,y,v_5\}$ induces a bull.  (A similar contradiction occurs
when $F$ is $F_4$, $F_5$ or $F_6$ and we omit the details.)
\end{proof}

One can test in polynomial time whether a graph contains any of $F_0$,
$F_1, \ldots,$ $F_6$.  It follows from Lemmas~\ref{lem:magnet}
and~\ref{lem:Fi} that if $G$ is a quasi-prime $(P_6,\mbox{bull})$-free
graph that contains no $K_5$ and no double wheel and contains any of
$F_0$, $F_1, \ldots, F_6$, then the $4$-colorability of $G$ can be
decided in polynomial time.  Therefore we will assume that $G$
contains none of $F_0$, $F_1, \ldots, F_6$.

%%%%%
\section{The gem-free case}\label{sec:gemfree}

In this section we examine what happens when we impose the additional
constraint that the graph is gem-free.  In this case we are able to
deal with the $k$-coloring problem for any $k$.  The main purpose of
this section is to give the proof of Theorem~\ref{thm:kp6bgf}.  Before
doing that we need some other results.

\begin{theorem}\label{thm:gemfree}
Let $G$ be a prime $(P_6,\mbox{bull}, \mbox{gem})$-free graph that
contains a $C_5$.  Then $G$ is triangle-free.
\end{theorem}
\begin{proof}
Since $G$ contains a $C_5$, there are five disjoint subsets $U_1,
\ldots, U_5$ of $V(G)$ such that the following properties hold for
each $i\in\{1,\ldots,5\}$, with subscripts modulo~$5$:
\begin{itemize}
\item
$U_i$ is anticomplete to $U_{i-2}\cup U_{i+2}$;
\item
$U_i$ contains a vertex that is complete to $U_{i-1}\cup U_{i+1}$.
\end{itemize}
Let $U=U_1\cup\cdots\cup U_5$ and $R=V(G)\setminus U$.  We choose
these sets so that the set $U$ is maximal with the above properties.
For each $i\in\{1,\ldots,5\}$ let $u_i$ be a vertex in $U_i$ that is
complete to $U_{i-1}\cup U_{i+1}$.  We observe that:
\begin{equation}\label{uis}
\mbox{Each of $U_1, \ldots, U_5$ is a stable set.}
\end{equation}
Proof: Suppose on the contrary and up to symmetry that $U_1$ is not a
stable set.  So $G[U_1]$ has a component $X$ of size at least $2$.
Since $G$ is prime, $X$ is not a homogeneous set, so there is a vertex
$z\in V(G)\setminus X$ and two vertices $x,y\in X$ such that $z$ is
adjacent to $y$ and not to $x$, and since $X$ is connected we may
choose $x$ and $y$ adjacent.  Suppose that $z$ is adjacent to $u_2$.
Then $z$ is adjacent to $u_5$, for otherwise $\{u_5,x,u_2,z,y\}$
induces a gem.  Then $z$ has no neighbor $v$ in $U_3$, for otherwise
$\{v,z,y,x,u_2\}$ induces a gem, and by symmetry $z$ has no neighbor
in $U_4$.  But then the $5$-tuple $(U_1\cup\{z\}$, $U_2$, $U_3$,
$U_4$, $U_5)$ contradicts the maximality of $U$ (since $u_2$ and $u_5$
are complete to $U_1\cup\{z\}$).  So $z$ is not adjacent to $u_2$,
and, by symmetry, $z$ is not adjacent to $u_5$.  Then $z$ is adjacent
to $u_3$, for otherwise $\{z,y,x,u_2,u_3\}$ induces a bull.  By
symmetry $z$ is adjacent to $u_4$.  But now $\{z,u_2, u_3, u_4, u_5\}$
induces a bull.  So (\ref{uis}) holds.

\medskip

It follows easily from the definition of the sets $U_1,\ldots,U_5$ and
(\ref{uis}) that $G[U]$ contains no triangle.  Moreover:
\begin{equation}\label{notuur}
\mbox{There is no triangle $\{x,y,z\}$ with $x,y\in U$ and $z\in R$.}
\end{equation}
Proof: Suppose the contrary.  By (\ref{uis}) and up to symmetry, let
$x\in U_1$ and $y\in U_2$.  Then $z$ is adjacent to exactly one of
$u_3,u_5$, for otherwise $\{u_5,x,y,z,u_3\}$ induces a bull or a gem.
Up to symmetry we may assume that $z$ is adjacent to $u_3$ and not to
$u_5$.  Then $z$ has no neighbor $v\in U_4$, for otherwise
$\{x,y,u_3,v,z\}$ induces a gem; and $z$ is adjacent to $u_1$, for
otherwise $\{u_1,y,z,u_3,u_4\}$ induces a bull; and $z$ has no
neighbor $v\in U_5$, for otherwise $\{v,x,y,u_3,z\}$ induces a gem.
It follows that the $5$-tuple $(U_1$, $U_2\cup\{z\}$, $U_3$, $U_4$,
$U_5)$ contradicts the maximality of $U$ (since $u_1$ and $u_3$ are
complete to $U_2\cup\{z\}$).  So (\ref{notuur}) holds.

\begin{equation}\label{noturr}
\mbox{There is no triangle $\{x,y,z\}$ with $x\in U$ and $y,z\in R$.}
\end{equation}
Proof: Suppose the contrary.  Up to symmetry, let $x\in U_1$.  Let $X$
be the component of $N(x)$ that contains $y,z$.  Since $G$ is prime,
$X$ is not a homogeneous set, so there is a vertex $t$ with a neighbor
and a non-neighbor in $X$, and since $X$ is connected and up to
relabelling we may assume that $t$ is adjacent to $y$ and not to~$z$.
Vertex $t$ is not adjacent to $x$, by the definition of $X$.  By
(\ref{notuur}), $y$ and $z$ have no neighbor in $\{u_2,u_5\}$.  Then
$t$ is adjacent to $u_2$, for otherwise $\{t,y,z,x,u_2\}$ induces a
bull, and by symmetry $t$ is adjacent to $u_5$.  If $t$ is adjacent to
$u_3$, then it is also adjacent to $u_4$, for otherwise
$\{x,u_2,t,u_3,u_4\}$ induces a bull; but then $\{u_2,u_3,u_4,u_5,t\}$
induces a gem.  So $t$ is not adjacent to $u_3$, and, by symmetry, $t$
is not adjacent to $u_4$.  If $y$ is adjacent to $u_3$, then $z$ is
adjacent to $u_3$, for otherwise $\{x,y,z,u_3,u_5\}$ induces a bull;
but then $\{u_3,y,z,u_4,t\}$ induces a bull.  So $y$ is not adjacent
to $u_3$, and also not to $u_4$ by symmetry, and similarly $z$ has no
neighbor in $\{u_3,u_4\}$.  But then $u_3$-$u_4$-$u_5$-$t$-$y$-$z$ is
an induced $P_6$, a contradiction.  So (\ref{noturr}) holds.

\begin{equation}\label{notrrr}
\mbox{There is no triangle $\{x,y,z\}$ with $x,y,z\in R$.}
\end{equation}
Proof: Suppose there is a such a triangle.  Since $G$ is prime it is
connected, so there is a shortest path $P$ from $U$ to a triangle
$T=\{x,y,z\}\subseteq R$.  Let $P=p_0$-$\cdots$-$p_k$, with $p_0\in
U$, $p_1, \ldots, p_k\in R$, and $p_k=x$, and $k\ge 1$.  We may assume
that $p_0\in U_1$.  We observe that $y$ is not adjacent to $p_{k-1}$,
for otherwise $\{x,y,p_{k-1}\}$ is a triangle and $P\setminus p_k$ is
a shorter path than $P$; and $y$ has no neighbor $p_i$ in $P\setminus
\{p_k,p_{k-1}\}$, for otherwise $p_0$-$\cdots$-$p_i$ is a shorter path
than $P$ from $U$ to $T$.  Likewise, $z$ has no neighbor in
$P\setminus p_k$.  Moreover there is no edge between
$P\setminus\{p_0,p_1\}$ and $U$ for otherwise there is a path strictly
shorter than $P$ between $U$ and $T$.  By (\ref{notuur}) $p_1$ has no
neighbor in $\{u_2,u_5\}$ and has at most one neighbor in
$\{u_3,u_4\}$; by symmetry we may assume that $p_1$ is not adjacent to
$u_4$.  If $k\ge 3$, then $u_4$-$u_5$-$p_0$-$p_1$-$p_2$-$p_3$ is an
induced $P_6$.  If $k=2$, then $u_4$-$u_5$-$p_0$-$p_1$-$p_2$-$y$ is an
induced $P_6$.  So $k=1$.  Then $p_1$ is adjacent to $u_3$, for
otherwise $u_3$-$u_4$-$u_5$-$p_0$-$p_1$-$y$ is an induced $P_6$.  Let
$X$ be the component of $N(p_1)$ that contains $y,z$.  Since $G$ is
prime, $X$ is not a homogeneous set, so there is a vertex $t$ with a
neighbor and a non-neighbor in $X$, and since $X$ is connected and up
to relabelling we may assume that $t$ is adjacent to $y$ and not to
$z$.  Vertex $t$ is not adjacent to $x$, by the definition of $X$.
Then $t$ is adjacent to $p_0$, for otherwise $\{t,y,z,p_1,p_0\}$
induces a bull, and $t$ is adjacent to $u_3$, for otherwise
$\{t,y,z,p_1,u_3\}$ induces a bull.  By (\ref{notuur}), $t$ has no
neighbor in $\{u_4,u_5\}$.  Then $u_5$-$u_4$-$u_3$-$t$-$y$-$z$ is an
induced $P_6$.  So (\ref{notrrr}) holds.

\medskip

Claims (\ref{uis})--(\ref{notrrr}) imply the theorem.
\end{proof}

\paragraph*{Clique-width}

The \emph{clique-width} of a graph $G$ is an integer parameter $w$
which measures the complexity of constructing $G$ through a sequence
of certain operations; such a sequence is called a
\emph{$w$-expression}.  This concept was introduced in \cite{CER}.  We
will not recall here all the technical definitions associated with
clique-width, and we will only mention the results that we use.  Say
that a class of graphs $\cal C$ has \emph{bounded clique-width} if
there is a constant $c$ such that every graph in $\cal C$ has
clique-width at most $c$.

\begin{theorem}[\cite{CMR}]\label{thm:cw0}
If a class of graphs $\cal C$ has bounded clique-width $c$, and there
is a polynomial $f$ such that for every graph $G$ in $\cal C$ with $n$
vertices and $m$ edges a $c$-expression can be found in time
$O(f(n,m))$, then for fixed $k$ the $k$-coloring problem can be solved
in time $O(f(n,m))$ for every graph $G$ in $\cal C$.
\end{theorem}

\begin{theorem}[\cite{CMR,CO}]\label{thm:cw1}
The clique-width of a graph that has non-trivial prime subgraphs is
the maximum of the clique-width of its prime induced subgraphs.
\end{theorem}

We will need two simple facts about $P_4$-free graphs.
\begin{theorem}\label{thm:p4f}
Every $P_4$-free graph satisfies the following two properties:
\begin{itemize}
\item
If $G$ has at least two vertices then it has a pair of twins
\cite{CLS}.
\item
$G$ has cliquewidth at most $2$ \cite{CO}.
\end{itemize}
\end{theorem}

Brandst\"adt et al.~\cite{BKM} studied $(P_6,K_3)$-free graphs and
established the following result.
\begin{theorem}[\cite{BKM}]\label{thm:BKM}
The class of $(P_6,K_3)$-free graphs has bounded clique-width $c$, and
a $c$-expression can be found in time $O(|V(G)|^2)$ for every graph
$G$ in this class.
\end{theorem}
In \cite{BKM} it is proved that $c\le 40$ and claimed that one can
obtain $c\le 36$.  The following theorem refers to the same constant
$c$.

\begin{theorem}\label{thm:bcw}
Let $G$ be a $(P_6,\mbox{bull},\mbox{gem})$-free graph that contains a
$C_5$.  Then $G$ has bounded clique-width $c$, and a $c$-expression
can be found in time $O(|V(G)|^2)$ for every graph $G$ in this class.
\end{theorem}
\begin{proof}
We may assume that $G$ is connected since the cliquewidth of a graph
is the maximum of the cliquewidth of its components.  Suppose that
$\overline{G}$ is not connected.  So $V(G)$ can be partitioned into
two non-empty sets $V_1$ and $V_2$ that are complete to each other.
Since $G$ is gem-free, each of $G[V_1]$ and $G[V_2]$ is $P_4$-free,
and consequently $G$ itself is $P_4$-free; so $G$ has cliquewith at
most $2$ by Theorem~\ref{thm:p4f}.  Therefore we may assume that $G$
and $\overline{G}$ are connected.  Let $M_1, \ldots, M_p$ be the
maximal modules of $G$.  Pick one vertex $m_i$ from each $M_i$, and
let $G'=G[\{m_1, \ldots, m_p\}]$.  Since $G$ and $\overline{G}$ are
connected we know from the theory of modular decomposition (see
Section~\ref{sec:mod}) that $M_1, \ldots, M_p$ form a partition of
$V(G)$, with $p\ge 4$, and that $G'$ is a prime graph.  Clearly $G'$
is $(P_6,\mbox{bull},\mbox{gem})$-free since it is an induced subgraph
of $G$.  We observe that:
\begin{equation}\label{mip4f}
\mbox{$G[M_i]$ is $P_4$-free, for each $i\in\{1,\ldots,p\}$.}
\end{equation}
Proof: Since $p\ge 2$ and $G$ is connected there is a module $M_j$
such that $j\neq i$ and $M_j$ is complete to $M_i$.  If $G[M_i]$
contains a $P_4$, then $m_j$ and the four vertices of this $P_4$
induce a gem, a contradiction.  So (\ref{mip4f}) holds.

\medskip

Consider any prime induced subgraph $H$ of $G$.  We claim that:
\begin{equation}\label{hmi}
\mbox{$H$ contains at most one vertex from each maximal module $M_i$.}
\end{equation}
Proof: Suppose that $H$ contains two vertices from some $M_i$.  By
(\ref{mip4f}) the subgraph of $G$ induced by $V(H)\cap M_i$ has a pair
of twins; but this contradicts the fact that $H$ is prime.  So
(\ref{hmi}) holds.

\medskip

By (\ref{hmi}), $H$ is isomorphic to an induced subgraph of $G'$.  By
Theorems~\ref{thm:gemfree} and~\ref{thm:BKM}, $H$ has bounded
clique-width.  Hence and by Theorem~\ref{thm:cw1}, $G$ has bounded
clique-width.
\end{proof}

\noindent{\it Proof of
Theorem~\ref{thm:kp6bgf}.}\setcounter{equation}{0} %
Let $G$ be a $(P_6,\mbox{bull},\mbox{gem})$-free graph.  Since $G$ is
$P_6$-free it contains no $C_\ell$ with $\ell\ge 7$, and since it is
gem-free it contains no $\overline{C_\ell}$ with $\ell\ge 7$.  So if
$G$ also contains no $C_5$, then it is a bull-free perfect graph.
In that case we can use the algorithms from either \cite{FM} or
\cite{Penev} to find a $\chi(G)$-coloring of $G$ in polynomial time,
and we need only check whether $\chi(G)\le k$.  (When $k=4$, we can do
a little better: by Lemmas~\ref{lem:magnet} and~\ref{lem:Fi} we may
assume that $G$ is also $F_5$-free, so $G$ contains no
$\overline{C_\ell}$ for any $\ell\ge 6$.  Then we can use the
algorithm from \cite{FMP}, which is simpler than those in
\cite{FM,Penev}.)

Now assume that $G$ contains a $C_5$.  Then Theorems~\ref{thm:bcw}
and~\ref{thm:cw0} imply that the $k$-coloring problem can be solved in
polynomial time.  {\hfill$\Box$\vspace{2ex}}

%%%%
\section{When there is a gem}\label{sec:yesgem}

By the results in the preceding sections, we may now focus on the case
when the graph contains a gem.  Suppose that $v_1, \ldots, v_5$ are
five vertices that induce a gem with edges $v_1v_2$, $v_2v_3$,
$v_3v_4$ and $v_5v_i$ for each $i\in\{1,2,3,4\}$.  We can define the
following sets.  Let $S=\{v_1, \ldots, v_5\}$ and let:
\begin{itemize}
\item
$V_i=\{x\in V(G)\mid$ $N_S(x)\setminus \{v_i\} =N_S(v_i)\}$ for each
$i\in\{1,\ldots,5\}$.
\item
$X=\{x\in V(G) \mid$ $x$ is complete to $\{v_1,v_4\}$ and anticomplete
to $\{v_2,v_3\}\}$.
\item
$W=\{x\in V(G) \mid$ $x$ is anticomplete to $\{v_1, v_2, v_3, v_4\}$
and has a neighbor in $V_5\}$.
\item
$Z=\{x\in V(G)\mid$ $x$ is anticomplete to $\{v_1,v_2, v_3,v_4\}\cup
V_5\}$.
\item
$Z_1 = \{x \in V(G) \mid$ $x$ is in any component of $Z$
that has a neighbor in $W\}$.
\item 
$Z_0 = Z \setminus Z_1$.
\end{itemize}
We note that constructing these sets can be done in time $O(n^2)$ by
scanning adjacency lists.
 
\begin{theorem}\label{thm:yesgem}
Let $G$ be a ($P_6$, bull)-free graph.  Assume that $G$ is
quasi-prime, contains no $K_5$, no double wheel and no $F_0$, $F_1$,
\ldots, $F_6$, and that $G$ contains a gem induced by $\{v_1, \ldots,
v_5\}$.  Let $S$, $V_i$ $(i=1,\ldots,5)$, $X$, $W$, $Z$, $Z_0$ and $Z_1$
be the sets defined as above.  Then the following holds:
\begin{enumerate}[label=(\alph*)]
\item\label{lxne}
$X$ is not empty.
\item\label{lx}
$X$ is anticomplete to $V_2\cup V_3\cup V_5$ and complete to
$V_1\cup V_4$.
\item\label{lvg}
$V(G) = \bigcup_{i=1}^5 V_i\cup W\cup X\cup Z$.
\item\label{lv5}
$V_5$ is complete to $V_1\cup\cdots\cup V_4$.
\item\label{lw}
$W$ is complete to $X$ and anticomplete to $V_1\cup\cdots\cup V_4$.
\item\label{lz}
$Z$ is anticomplete to $V_1\cup\cdots\cup V_4$.
\item\label{lz1}
$Z_1$ is complete to $X$.
\item\label{lcx}
Every component of $X$ is homogeneous and is a clique.
\item\label{lz0}
Every component of $Z_0$ is homogeneous and is a clique.
\item\label{lxh}
$X$ is a homogeneous set in $G\setminus Z_0$.
\item\label{lwz1}
If $Z_1\neq\emptyset$, then there is a vertex $w^*$ in $W$
such that $Z_1\cap N(w^*)$ is complete to $Z_1\setminus N(w^*)$.
\end{enumerate}
\end{theorem}

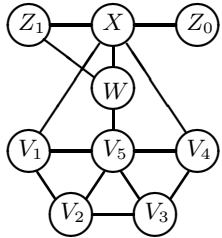
\begin{figure}[ht]
\unitlength=0.07cm
\thicklines
\begin{center}
\begin{picture}(40,44) %
   % vertices
\multiput(12,4)(16,0){2}{\ball}
\multiput(4,16)(16,0){3}{\ball}
\put(20,28){\ball}
\multiput(4,40)(16,0){3}{\ball}
   % edges
\put(16,4){\line(1,0){8}}
\multiput(8,16)(16,0){2}{\line(1,0){8}}
\multiput(8,40)(16,0){2}{\line(1,0){8}}
\multiput(20,20)(0,12){2}{\line(0,1){4}}
\multiput(14.9,6.9)(16,0){2}{\line(2,3){3.5}}
\multiput(9.1,6.9)(16,0){2}{\line(-2,3){3.5}}
\put(6.2,19.6){\line(2,3){11.2}}
\put(33.8,19.6){\line(-2,3){11.2}}
\put(6.9,37.2){\line(4,-3){9.5}}
\put(10,3){{\small $V_2$}}\put(26,3){{\small $V_3$}}
\put(2,15){{\small $V_1$}}\put(18,15){{\small 
$V_5$}}\put(34,15){{\small $V_4$}}
\put(18,26){{\small $W$}}
\put(2,39){{\small $Z_1$}}\put(18,39){{\small $X$}}\put(34,39){{\small $Z_0$}}
\end{picture}
\end{center}
\caption{The partition of $V(G)$ in Theorem~\ref{thm:yesgem}.
 A line between two sets represents partial or complete
adjacency.  No line means that the two sets are anticomplete to each
other.}
\label{fig:vg}
\end{figure}

\begin{proof}
Note that $v_i\in V_i$ for each $i\in\{1,\ldots,5\}$.  It is easy to
check from their definition that the sets $V_1, \ldots, V_5, X, W, Z$
are pairwise disjoint.  

\medskip

\ref{lxne} By Lemma~\ref{lem:magnet} we may assume that $G[S]$ itself
is not a magnet (and this can easily be checked in polynomial time).
Then the second item of Lemma~\ref{lem:Fi} implies the existence of a
vertex that is complete to $\{v_1,v_4\}$ and anticomplete to
$\{v_2,v_3,v_5\}$, so that vertex is in $X$.  Thus item~\ref{lxne}
holds.

\medskip

\ref{lx} Consider any $x\in X$.  Suppose that $x$ has a neighbor $v$
in $V_2\cup V_3\cup V_5$.  If $v\in V_5$, then $\{v_1,v_2,v_3,$ $v_4,$
$v,x\}$ induces an $F_2$, a contradiction.  So $x$ is anticomplete to
$V_5$; in particular $x$ is not adjacent to $v_5$.  If $v$ in $V_2$,
then $\{v_1,v, v_3,v_4,v_5,x\}$ induces an $F_4$.  The same holds if
$v\in V_3$.  If $x$ has a non-neighbor $u$ in $V_1$, then
$\{x,v_4,v_3,v_5,u\}$ induces a bull.  The same holds if $u\in V_4$.
Thus \ref{lx} holds.

\medskip

By~\ref{lxne} we pick a vertex $x_0\in X$.  By \ref{lx} $x_0$ is
not adjacent to $v_5$.

\medskip

\ref{lvg} Let $u$ be any vertex in $V(G)$.  First suppose that $u$ is
adjacent to both $v_1,v_4$.  Then $u$ has exactly one neighbor in
$\{v_2,v_3\}$, for otherwise $u$ is in $V_5$ or $X$.  So assume that
$u$ is adjacent to $v_2$ and not to $v_3$.  Then $u$ is adjacent to
$v_5$, for otherwise $\{u,v_1,\ldots,v_5\}$ induces an $F_4$, and to
$x_0$, for otherwise $\{x_0,v_1,u,v_5,v_3\}$ induces a bull; but then
$\{u,v_1,\ldots,v_5,x_0\}$ induces an $F_6$.  \\
Now suppose that $u$ is non-adjacent to both $v_1,v_4$.  Then $u$ has
exactly one neighbor in $\{v_2,v_3\}$, for otherwise either $u$ is in
$W\cup Z$ or $\{v_1,v_2,u,v_3,v_4\}$ induces a bull.  So assume that
$u$ is adjacent to $v_2$ and not to $v_3$; then $u$ is adjacent to
$v_5$, for otherwise $\{u,v_2,v_1,v_5,v_4\}$ induces a bull; and so
$u\in V_1$.  \\
Finally suppose, up to symmetry, that $u$ is adjacent to $v_1$ and not
to $v_4$.  If $u$ is not adjacent to $v_2$, then it is adjacent to
$v_5$, for otherwise $\{u,v_1,v_2,v_5,v_4\}$ induces a bull, and to
$v_3$, for otherwise $\{u,v_1,\ldots,v_5\}$ induces an $F_1$; and so
$u$ is in $V_2$.  So suppose that $u$ is adjacent to $v_2$.  if $u$ is
not adjacent to $v_3$, then it is adjacent to $v_5$, for otherwise
$\{u,v_1,\ldots,v_5\}$ induces an $F_3$; and so $u$ is in $V_1$.  So
suppose that $u$ is adjacent to $v_3$.  If $u$ is not adjacent to
$v_5$, then it is adjacent to $x_0$, for otherwise $\{u,v_3,v_5,
v_4,x_0\}$ induces a bull; but then $\{u,v_1,v_3,v_4,v_5,x_0\}$
induces an $F_5$.  So $u$ is adjacent to $v_5$, and so $u$ is in
$V_2$.  Thus \ref{lvg} holds.

\medskip

\ref{lv5} Consider any $v\in V_5$.  Suppose that $v$ has a
non-neighbor $u\in V_1\cup V_2$.  Note that $v\neq v_5$ and
$u\notin\{v_1,v_2\}$.  By \ref{lx}, $v$ is not adjacent to $x_0$.  If
$u\in V_1$, then $u$ is adjacent to $v_1$, for otherwise
$\{u,v_2,v_1,v,v_4\}$ induces a bull; but then
$\{u,v,v_1,v_2,v_3,v_4\}$ induces an $F_3$.  If $u\in V_2$, then,
by~\ref{lx}, $u$ is not adjacent to $x_0$; but then
$\{u,v_3,v,v_4,x_0\}$ induces a bull.  Thus \ref{lv5} holds.

\medskip

\ref{lw} Consider any $w\in W$.  By the definition of $W$, $w$ has a
neighbor $v$ in $V_5$.  Consider any $x\in X$.  By \ref{lx}, $v$ is
not adjacent to $x$.  Then $w$ is adjacent to $x$, for otherwise
$\{w,v,v_3,v_4,x\}$ induces a bull.  So $w$ is complete to $X$.  Now
suppose up to symmetry, that $w$ has a neighbor $u\in V_1\cup V_2$.
We know that $w$ is adjacent to $x_0$ as proved just above.  By
\ref{lx}, $x_0$ is not adjacent to $v$.  By \ref{lv5} $v$ is adjacent
to $u$.  If $u\in V_1$, then $\{u,v,w,v_2,v_3,v_4\}$ induces an $F_1$.
If $u\in V_2$, then $u$ is adjacent to $v_2$, for otherwise
$\{v_2,v,u,w,x_0\}$ induces a bull; but then $\{w,u,v_2,v_3,v_4\}$
induces a bull.  Thus \ref{lw} holds.

\medskip

\ref{lz} Suppose, up to symmetry, that some vertex $z\in Z$ has a
neighbor $u\in V_1\cup V_2$.  If $u\in V_1$, then
$\{z,u,v_2,v_5,v_4\}$ induces a bull.  If $u\in V_2$, then
$\{z,u,v_1,v_5,v_4\}$ induces a bull.  So \ref{lz} holds.

\medskip

\ref{lz1} Consider any $z\in Z_1$ and $x\in X$.  By the definition of
$Z_1$, there is a path $z_0$-$\cdots$-$z_\ell$ such that $z_0\in W$,
$z_1, \ldots, z_\ell\in Z_1$ and $z=z_\ell$.  We take a shortest such
path, so if $\ell\ge 2$ then $z_2,\ldots,z_\ell$ are not adjacent to
$z_0$.  By~\ref{lw} $x$ is adjacent to $z_0$.  Then by induction on
$i=1,\ldots,\ell$, and by \ref{lx} and \ref{lz}, we see that $x$ is
adjacent to $z_i$, for otherwise $z_i$-$z_{i-1}$-$x$-$v_1$-$v_2$-$v_3$
is an induced $P_6$.  Thus \ref{lz1} holds.

\medskip

\ref{lcx} Suppose that some component $Y$ of $X$ is not homogeneous;
so there are adjacent vertices $x,y\in Y$ and a vertex $z\in
V(G)\setminus Y$ such that $z$ is adjacent to $y$ and not to $x$.  By
\ref{lx}, \ref{lw} and \ref{lz1} we have $z\in Z_0$.  Then, by
\ref{lx}, $\{z,y,x,v_1,v_2\}$ induces a bull.  So $Y$ is homogeneous,
and consequently $Y$ is a clique since $G$ is quasi-prime.  Thus
\ref{lcx} holds.

\medskip

\ref{lz0} Suppose that some component $Y$ of $Z_0$ is not homogeneous;
so there are adjacent vertices $y,z\in Y$ and a vertex $x\in
V(G)\setminus Y$ such that $x$ is adjacent to $y$ and not to $z$.  By
\ref{lz} and the definition of $Z_0$ and $Y$, we have $x\in X$.  Then,
by \ref{lx}, $z$-$y$-$x$-$v_1$-$v_2$-$v_3$ is an induced $P_6$.  So
$Y$ is homogeneous, and consequently $Y$ is a clique since $G$ is
quasi-prime.  Thus \ref{lz0} holds.

\medskip

\ref{lxh} By \ref{lx}, \ref{lw} and \ref{lz1}, $X$ is complete to
$V_1\cup V_4\cup W\cup Z_1$ and anticomplete to $V_2\cup V_3\cup V_5$.
So \ref{lxh} holds.

\medskip

\ref{lwz1} By the definition of $Z_1$, some vertex $w^*$ in $W$ has a
neighbor in $Z_1$.  Suppose that $Z_1\cap N(w^*)$ is not complete to
$Z_1\setminus N(w^*)$.  So there are non-adjacent vertices $y\in
Z_1\cap N(w^*)$ and $z \in Z_1\setminus N(w^*)$.  By \ref{lw} and
\ref{lz1}, $x_0$ is complete to $\{w^*,y,z\}$.  By the definition of
$W$, $w^*$ has a neighbor $v$ in $V_5$.  By \ref{lx} and the
definition of $Z$, $v$ is anticomplete to $\{x_0,y,z\}$.  Then
$\{v,w^*,y,x_0,z\}$ induces a bull.  So \ref{lwz1} holds.  This
completes the proof of the lemma.
\end{proof}

\begin{theorem}\label{thm:yesgemcol}
Let $G$ be a ($P_6$, bull)-free graph.  Assume that $G$ is
quasi-prime, contains no $K_5$, no double wheel and no $F_0$, $F_1$,
\ldots, $F_6$, and that $G$ contains a gem.  Then we can determine in
polynomial time whether $G$ is $4$-colorable.
\end{theorem}

\begin{proof}
Let $v_1,\ldots,v_5$ be five vertices that induce a gem, with edges
$v_1v_2$, $v_2v_3$, $v_3v_4$ and $v_5v_i$ for each $i\in\{1,2,3,4\}$,
and let $V_i$ $(i=1,\ldots,5)$, $X$, $W$, $Z$, $Z_0$ and $Z_1$ be the
sets defined as in Theorem~\ref{thm:yesgem}.  In this proof all items
\ref{lxne}--\ref{lwz1} that we invoke refer to
Theorem~\ref{thm:yesgem}.  First, we observe that:
\begin{equation}\label{GmZ0}
\longbox{$G$ is $4$-colorable if and only if $G\setminus Z_0$ is
$4$-colorable.  Moreover, given any $4$-coloring of $G\setminus Z_0$
we can make a $4$-coloring of $G$ in polynomial time.}
\end{equation}
Proof: Clearly if $G$ is $4$-colorable then $G\setminus Z_0$ is
$4$-colorable.  So let us prove the converse and the second sentence
of the claim.  Let $c$ be a $4$-coloring of $G\setminus Z_0$.  Let $t$
be the maximum size of a component of $X$.  By item~\ref{lcx} we may
assume, up to relabeling, that the colors used on such a component are
$1,\ldots, t$.  By items~\ref{lx}, \ref{lw} and \ref{lz1}, these
colors are not used on $V_1\cup V_4\cup W\cup Z_1$.  So for every
component $Y$ of $X$ we can recolor the vertices of $Y$ with colors
$1,\ldots,|Y|$.  Thus we obtain a $4$-coloring $c'$ of $G\setminus
Z_0$.  Now we can extend $c'$ to $Z_0$ as follows.  Let $U$ be any
component of $Z_0$.  By the definition of $Z$ and $Z_0$ and by
item~\ref{lz}, we have $N(U)\subseteq X$.  Let $Y$ be the largest
component of $X$ that is adjacent to $U$.  By items~\ref{lcx} and
\ref{lz0}, $U$ and $Y$ are complete to each other and $U\cup Y$ is a
clique.  Since $G$ contains no $K_5$, we have $|U|\le 4-|Y|$.
Moreover, by the choice of $Y$, any component $Y'$ of $X$ that is
adjacent to $U$ is not larger than $Y$, so the colors used on $Y'$ are
also used on $Y$.  So $U$ can be colored with the colors from
$\{1,2,3,4\}$ that are not used on $Y$.  We can proceed similarly for
all $U$.  This yields a $4$-coloring of $G$.  Thus (\ref{GmZ0}) holds.

\medskip

By (\ref{GmZ0}) we may assume that $Z_0=\emptyset$.  By item~\ref{lxh}
we may assume that $X$ is a clique, with $|X|\le 3$.

\medskip

Now we can describe the coloring procedure.  We ``precolor'' a set $P$
of vertices (of size at most $8$), that is, we try every $4$-coloring
$f$ of $P$ and check whether the precoloring $f$ extends to a
$4$-coloring of $G$.  Each vertex $v$ in $V(G)\setminus P$ has a list
$L(v)$ of available colors, which consists of the set $\{1,2,3,4\}$
minus the colors given by $f$ to the neighbors of $v$ in $P$.  Hence
we want to solve the $L$-coloring problem on $G\setminus P$ or
determine that it has no solution.

\medskip

First suppose that $|X|\ge 2$.  Let $P=\{v_1, v_2, v_3, v_4, v_5\}\cup
X$.  So $|P|\le 8$.  It follows from items~\ref{lv5}, \ref{lw}, and
\ref{lz1} that every vertex in $V(G)\setminus P$ has two adjacent
neighbors in $P$.  So every vertex $v$ in $V(G)\setminus P$ satisfies
$|L(v)|\le 2$.  Hence checking whether $f$ extends to $G$ is a
$2$-list-coloring problem on the vertices of $G\setminus P$, which can
be solved in polynomial time.  Therefore we may assume that $|X|=1$,
and so $X=\{x_0\}$.

\medskip

Let $P=\{v_1,v_2,v_3,v_4,x_0\}$.  Clearly $|f(\{v_1,v_2,
v_3,v_4\})|\ge 2$; moreover we may assume that $|f(\{v_1,v_2,
v_3,v_4\})| \le 3$ for otherwise the precoloring cannot be extended to
$V_5$ and we stop examining it.  We distinguish two cases.

\medskip

{\it Case 1: $|f(\{v_1,v_2, v_3,v_4\})|= 3$.} \\
We may assume up to relabeling that $f(\{v_1,v_2, v_3,v_4\})=
\{1,2,3\}$.  Then $L(v)=\{4\}$ for all $v\in V_5$, so $V_5$ must be a
stable set, for otherwise the precoloring cannot be extended to $V_5$
and we stop examining it.  So let us assume that $V_5$ is a stable
set, and let $f(v)=4$ for all $v\in V_5$.  \\
Suppose that $f(x_0)=4$.  In that case we have $L(v)=\{1,2,3\}$ for
all $v\in W\cup Z$.  We can check whether $G[W\cup Z]$ is
$3$-colorable with the known algorithms \cite{RS,BFGP}.  On the other
hand we have $|L(u)|\le 2$ for all $u\in V_1\cup V_2\cup V_3\cup V_4$,
so checking whether $f$ extends to $V_1\cup V_2\cup V_3\cup V_4$ is a
$2$-list coloring problem.  By items~\ref{lw} and \ref{lz} the two
sets $V_1\cup V_2\cup V_3\cup V_4$ and $W\cup Z$ are anticomplete to
each other, so extending the coloring to them can be done
independently.  \\
Now suppose that $f(x_0)\neq 4$.  Then every vertex in $W$ has a list
of size $2$ (the set $\{1,2,3,4\}\setminus \{4,f(x_0)\}$).  If
$Z_1\neq\emptyset$, we pick a vertex $w^*$ from $W$ as in
item~\ref{lwz1} and add $w^*$ to $P$; moreover, if $w^*$ is not
complete to $Z_1$, we pick one vertex $z^*$ from $N_{Z_1}(w^*)$ and
add $z^*$ to $P$.  It follows from items~\ref{lv5}, \ref{lw},
\ref{lz1} and \ref{lwz1} that every vertex in $G\setminus P$ has a
list of size $2$ (in particular every vertex in $Z_1$ is complete to
either $\{x_0, w^*\}$ or $\{x_0, z^*\}$), so we can finish with a
$2$-list coloring problem.

\medskip

{\it Case 2:  $|f(\{v_1,v_2, v_3,v_4\})|= 2$.} \\
We may assume up to relabeling that $f(v_1)=f(v_3)=1$ and
$f(v_2)=f(v_4)=2$.  \\
Suppose that $V_1$ contains two adjacent vertices $a,b$.  Then
$\{a,b,v_2\}$ is a clique of size $3$.  We add $a,b$ to the set $P$.
By item~\ref{lv5}, in any possible $4$-coloring of $G$ the vertices of
$V_5$ must all have the same color, say color $4$.  In that case we
can argue as in Case~1 and conclude.  The same argument can be applied
if $V_2$ is not a stable set, and by symmetry if $V_3$ or $V_4$ is not
a stable set.  Therefore we may assume that each of $V_1, V_2, V_3,
V_4$ is a stable set.  \\
We have $L(v)=\{3,4\}$ for all $v\in V_5$, and we may assume, up to
symmetry, that $f(x_0)=4$.  So we have $L(v)=\{1,2,3\}$ for all $v\in
W\cup Z_1$ by items~\ref{lw}, \ref{lz}, and \ref{lz1}.  We may assume
that all vertices in $V_1\cup V_3$ receive color~$1$ and all vertices
in $V_2\cup V_4$ receive color~$2$, because the only other vertices
that may receive color $1$ or $2$ are in $W\cup Z_1$ and are
anticomplete to $V_1\cup V_2\cup V_3\cup V_4$.  Therefore we must only
extend the coloring to $V_5\cup W\cup Z_1$.  \\
Since $L(v)=\{3,4\}$ for all $v\in V_5$, the set $V_5$ must be
bipartite, for otherwise the precoloring cannot be extended to $V_5$
and we stop examining it.  So assume that $V_5$ is bipartite.  Let
$D_1,\ldots,D_t$ be the components of $V_5$ of size at least $2$
(which we call the \emph{big} components of $V_5$), if any.  For each
$D_i$, let $A_i,B_i$ be the two stable sets that form a partition of
$D_i$; let $W_{A_i}=\{x\in W\mid$ $x$ has a neighbor in $A_i$ and no
neighbor in $B_i\}$, $W_{B_i}=\{x\in W\mid$ $x$ has a neighbor in
$B_i$ and no neighbor in $A_i\}$, and $W_i=\{x\in W\mid$ $x$ has a
neighbor in each of $A_i$ and $B_i\}$.  We claim that:
\begin{equation}\label{uc}
\longbox{For every big component $D_i$ of $V_5$, each of $A_i$ and
$B_i$ contains a vertex that is complete to $W_i$.}
\end{equation}
Proof: Let $d$ be a vertex in $B_i$ (the proof is similar for $A_i$)
that has the most neighbors in $W_i$, and suppose that there is still
a vertex $u\in W_i$ that is not adjacent to $d$.  By the definition of
$W_i$ vertex $u$ has a neighbor $a$ in $A_i$ and a neighbor $b$ in
$B_i$.  In $D_i$ there is a shortest path $Q$ from $a$ to $b$, of odd
length.  It is easy to see that $D_i$ is $P_3$-connected.  If $a,b$
are adjacent, then Lemma~\ref{lem:P3conn} is contradicted by $D_i$,
$u$ and $x_0$, because $u$ is not adjacent to $d$.  So $a,b$ are not
adjacent, and since $G$ is $P_6$-free we have $Q=a$-$b'$-$a'$-$b$ for
some $a'\in A_i$ and $b'\in B_i$, and $u$ has no neighbor in
$\{a',b'\}$.  Then $d$ is adjacent to $a$, for otherwise
$\{u,a,b',v_1,d\}$ induces a bull, and $d$ is adjacent to $a'$, for
otherwise $\{u,a,d,v_1,a'\}$ induces a bull.  By the choice of $d$
some vertex $v$ in $W_i$ is adjacent to $d$ and not to $b$.  Then $v$
is adjacent to $a$, for otherwise $\{v,d,a,v_1,b\}$ induces a bull,
and to $a'$, for otherwise $\{x_0,v,a,d,a'\}$ induces a bull; but then
$\{x_0,v,d,a',b\}$ induces a bull, a contradiction.  Thus (\ref{uc})
holds.
\begin{equation}\label{wb}
\mbox{For every big component $D_i$ of $V_5$, one of $W_{A_i}$ and
$W_{B_i}$ is empty.}
\end{equation}
Proof: Suppose on the contrary that some vertex $u$ in $W$ has a
neighbor $a$ in $A_i$ and no neighbor in $B_i$ and some vertex $v$ in
$W$ has a neighbor $b$ in $B_i$ and no neighbor in $A_i$.  If $a,b$
are adjacent, then $u,v$ are adjacent, for otherwise $\{u,a,v_1,b,v\}$
induces a bull; but then $\{v_1,a,b,u,v,x_0\}$ induces an $F_5$.
Hence $a,b$ are not adjacent.  Since $G$ is $P_6$-free, $D_i$ contains
a chordless path $a$-$b'$-$a'$-$b$ of length $3$.  Then
$\{u,a,b',v_1,b\}$ induces a bull, a contradiction.  Thus (\ref{wb})
holds.

\medskip

By (\ref{wb}) we may assume that $W_{B_i}=\emptyset$ for every big
component $D_i$ of $V_5$.  For each big component $D_i$ of $V_5$, take
a vertex $d_i$ that is complete to $W_i$, with $d_i\in B_i$, which is
possible by (\ref{uc}); so $d_i$ is anticomplete to $W_{A_i}$.  Let
$T=\{d_1, \ldots, d_t\}$.  Note that $T$ is a stable set.  Let $H=
G[Z_1\cup W\cup T\cup \{v_1,v_2\}]$.  We claim that:
\begin{equation}\label{ft}
\mbox{$f$ extends to a $4$-coloring of $G$ if and only if $H$ is
$3$-colorable.}
\end{equation}
Proof: Suppose that $f$ extends to a $4$-coloring $c$ of $G$.  Clearly
every big component $D_i$ of $V_5$ satisfies either $c(A_i)=4$ and
$c(B_i)=3$ or vice-versa.  If every big component $D_i$ of $V_5$
satisfies $c(A_i)=4$ and $c(B_i)=3$, then the restriction of $c$ to
$H$ is a $3$-coloring, using colors $1,2,3$.  So suppose that some
component $D_i$ satisfies $c(A_i)=3$ and $c(B_i)=4$.  Then we swap
colors $3$ and $4$ on that component, and we claim that the result is
still a proper coloring.  Indeed, vertices in $V_1\cup V_2\cup V_3\cup
V_4$ have color $1$ or $2$; vertices in $W_i$ have a neighbor in each
of $A_i$ and $B_i$, so their color is $1$ or $2$; vertices in
$W_{A_i}$ do not have color $4$ since they are adjacent to $x_0$; and
all other vertices of $G$ are anticomplete to $D_i$, by the definition
of $Z$, $D_i$, $W_{A_i}$, $W_i$ and because $W_{B_i}=\emptyset$.  So
the swap does not cause any two adjacent vertices to have the same
color.  We can repeat this operation for every such component $D_i$;
thus we obtain a $4$-coloring of $G$ whose restriction to $H$ is a
$3$-coloring.  \\
Conversely, suppose that $H$ admits a $3$-coloring $g$, with colors
$1,2,3$.  Up to relabeling we may assume that $g(v_1)=1$ and
$g(v_2)=2$.  It follows that all vertices in $T$ have color $3$.  Then
we extend this coloring to $G$ as follows.  Assign color $1$ to all
vertices in $V_1\cup V_3$ and color $2$ to all vertices in $V_2\cup
V_4$.  For every big component $D_i$ of $V_5$, assign color $3$ to all
vertices in $B_i$ and color $4$ to all vertices in $A_i$.  Also assign
color $4$ to all vertices in $V_5\setminus (D_1\cup\cdots\cup D_t)$
and to $x_0$.  Thus we obtain a proper $4$-coloring $c$ of $G$, and
clearly $c$ is also an extension of $f$.  So (\ref{ft}) holds.

\medskip

By (\ref{ft}) we need only check whether the induced subgraph $H$ is
$3$-colorable, which we can do with the known algorithms
\cite{RS,BFGP}.  This completes the proof of the theorem.
\end{proof}

The time complexity of the coloring algorithm given in
Theorem~\ref{thm:yesgemcol} can be evaluated as follows.  We test only
a fixed number of precolorings, and for each of them we need to solve
either a list-$2$-coloring problem, which takes time $O(n^2)$, or the
problem of $3$-coloring a certain $P_6$-free subgraph of $G$, which
takes time $O(n^3)$ in \cite{BFGP}.  So the complexity is $O(n^3)$.

\medskip

Finally, Theorem~\ref{thm:main} follows from Lemmas~\ref{lem:abc},
\ref{lem:magnet} and~\ref{lem:Fi} and Theorems~\ref{thm:kp6bgf}
and~\ref{thm:yesgemcol}.

\medskip

The complexity of our general algorithm can be evaluated as follows.
Assume that we are given a $(P_6,\mbox{bull})$-free graph $G$ on $n$
vertices.  We first apply the reduction steps described in
Lemma~\ref{lem:abc}; the complexity is $O(n^5)$ as discussed after the
proof of this lemma.  Then we test in time $O(n^5)$ whether $G$
contains a gem.  Suppose that $G$ is gem-free.  Then we test whether
$G$ is perfect, which in this case is equivalent to testing whether
$G$ is $C_5$-free and takes time $O(n^5)$.  If $G$ is perfect, we use
the algorithm from \cite{Penev} to compute the chromatic number of $G$
in time $O(n^6)$.  If $G$ is not perfect, we use the algorithm from
\cite{BKM}, based on the fact that the clique-width is bounded, which
runs in $O(n^2)$.  Finally, if the graph contains a gem, then we
construct in time $O(n^2)$ the partition as in
Theorem~\ref{thm:yesgem} and apply the method described in
Theorem~\ref{thm:yesgemcol}, which takes time $O(n^3)$.  Hence the
overall complexity is $O(n^6)$.

\section{Concluding remarks}

Our algorithm provides a $4$-coloring if the input graph $G$ is indeed
$4$-colorable, and otherwise it stops (with the message that $G$ is
not $4$-colorable).  In the latter case the algorithm does not exhibit
a certificate of non-$4$-colorability, in other words a $5$-critical
subgraph of $G$.  Moreover we do not know if there are only finitely
many $5$-critical $(P_6,\mbox{bull})$-free graphs.  It is an
interesting question to know wether it is possible to produce the list
of all $5$-critical $(P_6,\mbox{bull})$-free graphs.  One obstacle to
obtaining a solution of this question with our method is that when we
reduce the problem to an instance of the list-$2$-coloring problem it
seems difficult to translate a negative answer back in terms of the
presence of a certain subgraph.  The same situation occurs in
\cite{RS} and \cite{BFGP}, and the list of all $4$-critical $P_6$-free
graphs was determined only in \cite{CGSZ} (see also
\cite{BHS,HMRSV,MM} for critical $P_5$-free graphs).  Maybe the new
method described in \cite{GS} could help determine the list of
$5$-critical $(P_6,\mbox{bull})$-free graphs.

%\clearpage
% \small
%\section*{References}

\end{document}